\documentclass[fleqn,preprint,3p,a4paper]{elsarticle}

\usepackage{amssymb}
\usepackage{amsmath}
\usepackage{amsthm}
\usepackage{dcolumn}
\usepackage{endnotes}
\usepackage{tabularx}
\usepackage[matrix,arrow]{xy}
\usepackage{wasysym}
\usepackage{color}
\newtheorem{theorem}{Theorem}[section]
\newtheorem{proposition}[theorem]{Proposition}
\newtheorem{lemma}[theorem]{Lemma}
\newtheorem{corollary}[theorem]{Corollary}

\theoremstyle{definition}
\newtheorem{definition}[theorem]{Definition}
\newtheorem{example}[theorem]{Example}
\newtheorem{remark}[theorem]{Remark}
\newtheorem{question}[theorem]{Question}

\newcommand{\ir}{{\mathsf{Irr}}}

\newcommand{\cl}{{\rm cl}}
\newcommand{\ii}{{\rm int}}

\newcommand{\ua}{\mathord{\uparrow}}
\newcommand{\da}{\mathord{\downarrow}}

\newcommand{\mk}{\mathord{\mathsf{K}}}

\journal{Topology and its applications}

\begin{document}

\begin{frontmatter}



\title{Coincidence of the upper Vietoris topology and the Scott topology \tnoteref{t1}}
\tnotetext[t1]{This research was supported by the National Natural Science Foundation of China (Nos. 11661057, 11971287) and NSF of Jiangxi Province (20192ACBL20045)}


\author[X. Xu]{Xiaoquan Xu\corref{mycorrespondingauthor}}
\cortext[mycorrespondingauthor]{Corresponding author}
\ead{xiqxu2002@163.com}
\address[X. Xu]{School of Mathematics and Statistics,
Minnan Normal University, Zhangzhou 363000, China}
\author[Z. Yang]{Zhongqiang Yang}
\address[Z. Yang]{Department of Mathematics, Shantou University, Shantou 515063,
China}
\ead{zqyang@stu.edu.cn}

\begin{abstract}

For a $T_0$ space $X$, let $\mk (X)$ be the poset of all compact saturated sets of $X$ with the reverse inclusion order. The space $X$ is said to have property Q if for any $K_1, K_2\in \mk (X)$, $K_2\ll K_1$ in $\mk (X)$ if{}f $K_2\subseteq \ii~\!K_1$. In this paper, we give several connections among the well-filteredness of $X$, the sobriety of $X$, the local compactness of $X$, the core compactness of $X$, the property Q of $X$, the coincidence of the upper Vietoris topology and Scott topology on $\mk (X)$, and the continuity of $x\mapsto\ua x : X \longrightarrow \Sigma~\!\! \mk (X)$ (where $\Sigma~\!\! \mk (X)$ is the Scott space of $\mk (X)$). It is shown that for a well-filtered space $X$ for which its Smyth power space $P_S(X)$ is first-countable, the following three properties are equivalent: the local compactness of $X$, the core compactness of $X$ and the continuity of $\mk (X)$. It is also proved that for a first-countable $T_0$ space $X$ in which the set of minimal elements of $K$ is countable for any compact saturated subset $K$ of $X$, the Smyth power space $P_S(X)$ is first-countable. For the Alexandroff double circle $Y$, which is Hausdorff and first-countable, we show that its Smyth power space $P_S(Y)$ is not first-countable.

\end{abstract}





\end{frontmatter}


\section{Introduction}

In non-Hausdorff topology and domain theory, we encounter numerous links between topology and order theory. There are a lot of connections among well-filteredness, sobriety, local compactness and core compactness. The Hofmann-Mislove Theorem, the spectral theory of distributive continuous lattices and the duality theorem of continuous semilattices show some of the most important such connections (see \cite{redbook, Jean-2013}). For a $T_0$ space $X$, let $\mk (X)$ be the poset of all compact saturated sets of $X$ with the reverse inclusion order. The space $X$ is said to have \emph{property Q} if for any $K_1, K_2\in \mk (X)$, $K_2\ll K_1$ in $\mk (X)$ if{}f $K_2\subseteq \ii~\!K_1$ (cf. \cite[Proposition I-1.24.2 and Proposition IV-2.19]{redbook}). It is well-known that the local compactness of a $T_0$ space $X$ implies that its topology $\mathcal O(X)$ is a continuous lattice. The spectral theory of continuous lattices shows that a sober space $X$ for which $\mathcal O(X)$ is a continuous domain is locally compact, and if a well-filtered space $X$ is locally compact, then $\mk (X)$ is a continuous semilattice, but the converse fails in general. The  duality theorem of continuous semilattices shows that for a sober space with property Q, $\mk (X)$ is a continuous semilattice iff $X$ is locally compact (see \cite{Hofmann-Mislove, Hofmann-Lawson, redbook, Jean-2013}). Thus the Lawson dual of $\mk (X)$ may be properly "bigger" than $\mathcal O(X)$.

The Smyth power spaces are very important structures in domain theory, which play a fundamental role in modeling the semantics of non-deterministic programming languages. There naturally arises a question of which topological properties are preserved by the Smyth power spaces. It was proved by Schalk \cite{Schalk} that the Smyth power space $P_S(X)$ of a sober space $X$ is sober (see also \cite[Theorem 3.13]{Klause-Heckmann}), and the upper Vietoris topology (that is, the topology of Smyth power space) agrees with the Scott topology on $\mk (X)$ if $X$ is a locally compact sober space. Xi and Zhao \cite{xi-zhao-MSCS-well-filtered} showed that a $T_0$ space $X$ is well-filtered iff $P_S(X)$ is a $d$-space. Recently, Brecht and Kawai \cite{Brecht} pointed out that $P_S(X)$ is second-countable for a second-countable $T_0$ space $X$, and the first author and Zhao \cite{xuxizhao} proved that a $T_0$ space $X$ is well-filtered iff $P_S(X)$ is well-filtered.

In this paper, we investigate some further connections among well-filteredness, sobriety, local compactness, core compactness, property Q, and coincidence of the upper Vietoris topology and the Scott topology. Especially, for a $T_0$ space $X$, we discuss the following questions:
\begin{enumerate}[\rm (a)]
\item Under what conditions does the core compactness of a $T_0$ space $X$ imply the local compactness of $X$? Does the continuity of $\mk (X)$ of $X$ equate to the local compactness of $X$?

\item When does the upper Vietoris topology and the Scott topology on $\mk (X)$ coincide?

\item Is the Smyth power space $P_S(X)$ of a first-countable $T_0$ space again first-countable?
\end{enumerate}

For a $T_0$ space $X$, we give several connections among the well-filteredness of $X$, the sobriety of $X$, the local compactness of $X$, the core compactness of $X$, the property Q of $X$, the coincidence of the upper Vietoris topology and Scott topology on $\mk (X)$, and the continuity of $x\mapsto\ua x : X \longrightarrow \Sigma~\!\! \mk (X)$ (where $\Sigma~\!\!\mk (X)$ is the Scott space of $\mk (X)$). It is shown that for a well-filtered space $X$ for which $P_S(X)$ is first-countable, the following three properties are equivalent: the local compactness of $X$, the core compactness of $X$ and the continuity of $\mk (X)$. Some known results are improved. It is proved that for a first-countable $T_0$ space $X$ in which the set of minimal elements of $K$ is countable for any compact saturated subset $K$ of $X$, the Smyth power space $P_S(X)$ is first-countable. For the Alexandroff double circle $Y$, which is Hausdorff and first-countable, we show that its Smyth power space $P_S(Y)$ is not first-countable.

\section{Preliminary}

In this section, we briefly recall some fundamental concepts and notations that will be used in the paper. Some basic properties of sober spaces, metric spaces and compact saturated sets are presented. For further details, we refer the reader to \cite{redbook, Jean-2013, Schalk}.

For a poset $P$ and $A\subseteq P$, let
$\mathord{\downarrow}A=\{x\in P: x\leq  a \mbox{ for some }
a\in A\}$ and $\mathord{\uparrow}A=\{x\in P: x\geq  a \mbox{
	for some } a\in A\}$. For  $x\in P$, we write
$\mathord{\downarrow}x$ for $\mathord{\downarrow}\{x\}$ and
$\mathord{\uparrow}x$ for $\mathord{\uparrow}\{x\}$. A subset $A$
is called a \emph{lower set} (resp., an \emph{upper set}) if
$A=\mathord{\downarrow}A$ (resp., $A=\mathord{\uparrow}A$). Define $A^{\uparrow}=\{x\in P : x \mbox{ is an upper bound of } A \mbox{ in }P\}$. Dually, define $A^{\downarrow}=\{x\in P : x \mbox{ is a lower bound of } $A$ \mbox{ in } $P$\}$. The set $A^{\delta}=(A^{\ua})^{\da}$ is called the \emph{cut} generated by $A$. Let $P^{(<\omega)}=\{F\subseteq P : F \mbox{~is a nonempty finite set}\}$, $P^{(\leqslant\omega)}=\{F\subseteq P : F \mbox{~is a nonempty countable set}\}$ and $\mathbf{Fin} ~P=\{\uparrow F : F\in P^{(<\omega)}\}$.
 For a nonempty subset $A$ of $P$, define $\mathrm{min} (A)=\{a\in A : a$ is a minimal element of  $A\}$.  For a set $X$ and $A, B\subseteq X$, $A\subset B$ means that $A\subseteq B$ but $A\neq B$, that is, $A$ is a proper subset of $B$.

A nonempty subset $D$ of a poset $P$ is \emph{directed} if every two
elements in $D$ have an upper bound in $D$. The set of all directed sets of $P$ is denoted by $\mathcal D(P)$. $I\subseteq P$ is called an \emph{ideal} of $P$ if $I$ is a directed lower subset of $P$. Let $\mathrm{Id} (P)$ be the poset (with the order of set inclusion) of all ideals of $P$. Dually, we define the concept of \emph{filters} and denote the poset of all filters of $P$ by $\mathrm{Filt}(P)$. A filter of $P$ is called \emph{principal} if it has a minimum element, that is, there is $x\in P$ with $F=\ua x$. $P$ is called a
\emph{directed complete poset}, or \emph{dcpo} for short, provided that $\bigvee D$ exists in $P$ for any
$D\in \mathcal D(P)$. $P$ is called \emph{bounded complete} if $P$ is a dcpo and $\bigwedge A$ exists in $P$ for any nonempty subset $A$ of $P$.

As in \cite{redbook}, the \emph{lower topology} on a poset $Q$, generated
by the complements of the principal filters of $Q$, is denoted by $\omega (Q)$. A subset $U$ of $Q$ is \emph{Scott open} if
(i) $U=\mathord{\uparrow}U$ and (ii) for any directed subset $D$ for
which $\bigvee D$ exists, $\bigvee D\in U$ implies $D\bigcap
U\neq\emptyset$. All Scott open subsets of $Q$ form a topology.
This topology is called the \emph{Scott topology} on $Q$ and
denoted by $\sigma(Q)$. The space $\Sigma~\!\! Q=(Q,\sigma(Q))$ is called the
\emph{Scott space} of $Q$. The topology generated by $\omega (Q)\bigcup \sigma (Q)$ is called the \emph{Lawson topology} on $Q$ and denoted by $\lambda (Q)$.

For a $T_0$ space $X$, we use $\leq_X$ to represent the \emph{specialization order} of $X$, that is, $x\leq_X y$ if{}f $x\in \overline{\{y\}}$). In the following, when a $T_0$ space $X$ is considered as a poset, the order always refers to the specialization order if no other explanation. Let $\mathcal O(X)$ (resp., $\Gamma (X)$) be the set of all open subsets (resp., closed subsets) of $X$. A space $X$ is \emph{locally hypercompact} (see \cite{E_20182, Heckmann}) if for each $x\in X$ and each open neighborhood $U$ of $x$, there is $ F\in X^{(<\omega)}$ such that $x\in\ii~\!\,\ua F\subseteq\ua F\subseteq U$. Let $|X|$ be the cardinality of $X$ and $\omega=|\mathbb{N}|$, where $\mathbb{N}$ is the set of all natural numbers

A $T_0$ space $X$ is called a $d$-\emph{space} (or \emph{monotone convergence space}) if $X$ (with the specialization order) is a dcpo
 and $\mathcal O(X) \subseteq \sigma(X)$ (cf. \cite{redbook, Wyler}). Obviously, for a dcpo $P$, $\Sigma~\!\! P$ is a $d$-space. A nonempty subset $A$ of a $X$ is \emph{irreducible} if for any $\{F_1, F_2\}\subseteq \Gamma (X)$, $A \subseteq F_1\bigcup F_2$ implies $A \subseteq F_1$ or $A \subseteq  F_2$.  Denote by $\ir(X)$ (resp., $\ir_c(X)$) the set of all irreducible (resp., irreducible closed) subsets of $X$. Clearly, every subset of $X$ that is directed under $\leq_X$ is irreducible. A space $X$ is called \emph{sober}, if for any  $F\in\ir_c(X)$, there is a unique point $a\in X$ such that $F=\overline{\{a\}}$. Clearly, sober spaces are $T_0$.

For a dcpo $P$ and $A, B\subseteq P$, we say $A$ is \emph{way below} $B$, written $A\ll B$, if for each $D\in \mathcal D(P)$, $\bigvee D\in \ua B$ implies $D\bigcap \ua A\neq \emptyset$. For $B=\{x\}$, a singleton, $A\ll B$ is
written $A\ll x$ for short. For $x\in P$, let $w(x)=\{F\in P^{(<\omega)} : F\ll x\}$, $\Downarrow x = \{u\in P : u\ll x\}$ and $K(P)=\{k\in P : k\ll k\}$. Points in $K(P)$ are called \emph{compact} elements of $P$.

For the following definition and related conceptions, please refer to \cite{redbook, Jean-2013}.

\begin{definition}\label{continuous domain etc} Let $P$ be a dcpo and $X$ a $T_0$ space.
\begin{enumerate}[\rm (1)]
\item $P$ is called a \emph{continuous domain}, if for each $x\in P$, $\Downarrow x$ is directed
and $x=\bigvee\Downarrow x$.
\item  $P$ is called an \emph{algebraic domain}, if for each $x\in P$, $K(P)\bigcap \da x$ is directed
and $x=\bigvee K(P)\bigcap \da x$.
\item $P$ is called a \emph{quasicontinuous domain}, if for each $x\in P$, $\{\ua F : F\in w(x)\}$ is filtered and $\ua x=\bigcap
\{\ua F : F\in w(x)\}$.
\item $X$ is called \emph{core compact} if $\mathcal O(X)$ is a \emph{continuous lattice}.
\end{enumerate}
\end{definition}

\begin{lemma}\label{cont=CD}\emph{(\cite{redbook})} For a dcpo $P$, $P$ is continuous if{}f for each $x\in U\in \sigma (P)$, there is $u\in U$ such that $x\in \ii~\!_{\sigma (P)}\ua u\subseteq \ua u\subseteq U$.
\end{lemma}

\begin{lemma}\label{quasic=LHC}\emph{(\cite{quasicont, Heckmann})} Let $P$ be a dcpo $P$. Then
\begin{enumerate}[\rm (1)]
\item $P$ is quasicontinuous if{}f $\Sigma ~\!\! P$ is locally hypercompact.
\item If $P$ is a quasicontinuous domain, then $\Sigma ~\!\! P$ is sober.
\end{enumerate}
\end{lemma}

A subset $B$ of a $T_0$ space $X$ is called \emph{saturated} if $B$ equals the intersection of all open sets containing it (equivalently, $B$ is an upper set in the specialization order). We shall use $\mk(X)$ to
denote the set of all nonempty compact saturated subsets of $X$ and endow it with the \emph{Smyth preorder}, that is, for $K_1,K_2\in \mathord{\mathsf{K}}(X)$, $K_1\sqsubseteq K_2$ if{}f $K_2\subseteq K_1$. Let $\mathcal S^u(X)=\{\uparrow x : x\in X\}$.

\begin{lemma}\label{sups in Smyth}\emph{(\cite{redbook})}  Let $X$ be a $T_0$ space. For a nonempty family $\{K_i : i\in I\}\subseteq \mk (X)$, $\bigvee_{i\in I} K_i$ exists in $\mk (X)$ if{}f~$\bigcap_{i\in I} K_i\in \mk (X)$. In this case $\bigvee_{i\in I} K_i=\bigcap_{i\in I} K_i$.
\end{lemma}

A topological space $X$ is called \emph{well-filtered} if $X$ is $T_0$, and for any open set $U$ and any $\mathcal{K}\in \mathcal D(\mathord{\mathsf{K}}(X))$, $\bigcap\mathcal{K}{\subseteq} U$ implies $K{\subseteq} U$ for some $K{\in}\mathcal{K}$.

We have the following implications (which can not be reversed):
\begin{center}
sobriety $\Rightarrow$ well-filteredness $\Rightarrow$ $d$-space.
\end{center}

For a $T_0$ space $X$, let $\mathrm{OFilt(\mathcal O(X))}=\sigma (\mathcal O(X))\bigcap \mathrm{Filt}(\mathcal O(X))$. $\mathcal{U}\subseteq \mathcal O(X)$ is called an \emph{open filter} if $\mathcal U\in \mathrm{OFilt(\mathcal O(X))}$. For $K\in \mk (X)$, let $\Phi (K)=\{U\in \mathcal O(X) : K\subseteq U\}$. Then $\Phi (K)\in \mathrm{OFilt(\mathcal O(X))}$ and $K=\bigcap \Phi (K)$. Obviously, $\Phi : \mk (X) \longrightarrow \mathrm{OFilt(\mathcal O(X))}, K\mapsto \Phi (K)$, is an order embedding.

The single most important result about sober spaces is the Hofmann-Mislove Theorem (see \cite{Hofmann-Mislove} or \cite[Theorem II-1.20 and Theorem II-1.21]{redbook}).

\begin{theorem}\label{Hofmann-Mislove theorem} \emph{(The Hofmann-Mislove Theorem)} For a $T_0$ space $X$, the following conditions are equivalent:
\begin{enumerate}[\rm (1)]
            \item $X$ is a sober space.
            \item  For any $\mathcal F\in \mathrm{OFilt}(\mathcal O(X))$, there is $K\in \mk (X)$ such that $\mathcal F=\Phi (K)$.
            \item  For any $\mathcal F\in \mathrm{OFilt}(\mathcal O(X))$, $\mathcal F=\Phi (\bigcap \mathcal F)$.
\end{enumerate}
\end{theorem}

By the Hofmann-Mislove Theorem, a $T_0$ space $X$ is sober iff $\Phi : \mk (X) \longrightarrow \mathrm{OFilt(\mathcal O(X))}$ is an order isomorphism.

\begin{theorem}\label{SoberLC=CoreC}\emph{(\cite{redbook, Jean-2013, Kou})}  For a $T_0$ space $X$, the following conditions are equivalent:
\begin{enumerate}[\rm (1)]
	\item $X$ is locally compact and sober.
	\item $X$ is locally compact and well-filtered.
	\item $X$ is core compact and sober.
\end{enumerate}
\end{theorem}

For $U\in \mathcal O(X)$, let $\Box U=\{K\in \mk (X) : K\subseteq  U\}$. The \emph{upper Vietoris topology} on $\mk (X)$ is the topology generated by $\{\Box U : U\in \mathcal O(X)\}$ as a base, and the resulting space is called the \emph{Smyth power space} or \emph{upper space} of $X$ and is denoted by $P_S(X)$ (cf. \cite{Heckmann, Schalk}).

\begin{remark}\label{xi embdding} (\cite{Heckmann, Klause-Heckmann, Schalk}) Let $X$ be a $T_0$ space. Then
\begin{enumerate}[\rm (1)]
\item The specialization order on $P_S(X)$ is the Smyth order (that is, $\leq_{P_S(X)}=\sqsubseteq$).
\item The \emph{canonical mapping} $\xi_X: X\longrightarrow P_S(X)$, $x\mapsto\ua x$, is an order and topological embedding.
\item $P_S(\mathcal S^u(X))$ is a subspace of $P_S(X)$ and $X$ is homeomorphic to $P_S(\mathcal S^u(X))$.
\end{enumerate}
\end{remark}

For a nonempty subset $C$ of a $T_0$ $X$, it is easy to see that $C$ is compact if{}f $\ua C\in \mk (X)$. Furthermore, we have the following useful result (see, e.g., \cite[pp.2068]{E_2009}).

\begin{lemma}\label{COMPminimalset} Let $X$ be a $T_0$ space and $C\in \mk (X)$. Then $C=\ua \mathrm{min}(C)$ and  $\mathrm{min}(C)$ is compact.
\end{lemma}

For a metric space $(X, d)$, $x\in X$ and a positive number $r$, let $B(x, \varepsilon)=\{y\in Y : d(x, y)<r\}$ be the $r$-\emph{ball} about $x$. For a set $A\subseteq X$ and a positive number $r$, by the $r$-\emph{ball} about $A$ we mean the set $B(A, r)=\bigcup_{a\in A}B(a, r)$.

The following two results are well-known (cf. \cite{Engelking}).

\begin{proposition}\label{metric space property}
Every metric space is perfectly normal and first-countable. Therefore, it is sober.
\end{proposition}

\begin{proposition}\label{metric space compact sets}  Let $(X, d)$ be a metric space and $K$ a compact set of $X$. Then for any open set $U$ containing $K$, there is an $r>0$ such that $K\subseteq B(K, r)\subseteq U$.
\end{proposition}

\section{Well-filtered spaces and locally compact spaces}

Firstly, we give the following two known results (see, e.g., \cite{redbook, Jean-2013, xi-zhao-MSCS-well-filtered, xu-shen-xi-zhao1}).

\begin{lemma}\label{WF V less S}  Let $X$ be a well-filtered space. Then
 \begin{enumerate}[\rm (1)]
 \item For any $\mathcal K\in \mathcal D(\mk (X))$, $\bigcap\mathcal K\in\mk (X)$ and $\bigvee_{\mk (X)}\mathcal K=\bigcap\mathcal K$.
 \item $P_S(X)$ is a $d$-space, and hence the upper Vietoris topology is coarser than the Scott topology on $\mk (X)$.
\end{enumerate}
\end{lemma}

\begin{lemma}\label{K(X) way below relation}  Let $X$ be a $T_0$ space. Then
 \begin{enumerate}[\rm (1)]
 \item $\mk (X)$ is semilattice \emph{(}the semilattice operation being $\bigcup$\emph{)}.
 \item Let $K_1, K_2\in \mk (X)$ and consider the following assertions:
 \begin{enumerate}[\rm (a)]
 \item $K_2\subseteq \ii~\! ~K_1$.
 \item $K_1\ll K_2$ in $\mk (X)$.
\end{enumerate}
\end{enumerate}
If $X$ is well-filtered, then \emph{(a)} $\Rightarrow$ \emph{(b)}, and if $X$ is locally compact, then \emph{(b)} $\Rightarrow$ \emph{(a)}.
 \begin{enumerate}[\rm (3)]
\item If $X$ is well-filtered and locally compact, then $\mk (X)$ is a continuous semilattice.
\end{enumerate}
\end{lemma}

\begin{definition}\label{property Q} A $T_0$ space $X$ is said to have \emph{property Q} if for any $K_1, K_2\in \mk (X)$, $K_2\ll K_1$ in $\mk (X)$ if{}f $K_2\subseteq \ii~\!K_1$.
\end{definition}

It follows from Lemma \ref{K(X) way below relation} that every locally compact well-filtered space has property {\rm Q}. Theorem \ref{main result 1} below shows that for a well-filtered space $X$ with the property {\rm Q}, $X$ is locally compact iff $\mk (X)$ is a continuous semilattice.

The following result is a direct inference of the Hofmann-Mislove Theorem.

\begin{proposition}\label{K(X) way below relation=Open case}\emph{(\cite{redbook})}  Let $X$ be a sober space. If $\mk (X)$ is continuous, then the following two conditions are equivalent:
 \begin{enumerate}[\rm (1)]
 \item $X$ has property Q.
 \item For any pair $(\mathcal U, \mathcal V)\in \sigma (\mathcal O(X))\times \sigma (\mathcal O(X))$, $\mathcal U\ll \mathcal V$ in $\sigma (\mathcal O(X))$ implies there is $V\in \mathcal V$ such that $V\subseteq \bigcap\mathcal U$.
\end{enumerate}
\end{proposition}

\begin{theorem}\label{CI+WF is sober} \emph{(\cite{xu-shen-xi-zhao2})}
	Every first-countable well-filtered space is sober.
	\end{theorem}

\begin{theorem}\label{WF core-compact is sober} \emph{(\cite{Lawson-Xi, xu-shen-xi-zhao1, xu-shen-xi-zhao2})} Every core compact well-filtered space is sober.
\end{theorem}

\begin{corollary}\label{WFcorcomp-sober}  A well-filtered space is locally compact if{}f it is core compact.
\end{corollary}

By Theorem \ref{WF core-compact is sober}, Theorem \ref{SoberLC=CoreC} can be strengthened into the following one.

\begin{theorem}\label{SoberLC=CoreCNew}  For a $T_0$ space $X$, the following conditions are equivalent:
\begin{enumerate}[\rm (1)]
	\item $X$ is locally compact and sober.
	\item $X$ is locally compact and well-filtered.
	\item $X$ is core-compact and sober.
    \item $X$ is core compact and well-filtered.
\end{enumerate}
\end{theorem}

Now we give one of the main results of this paper.

\begin{theorem}\label{main result 1}  For a well-filtered space $X$, the following conditions are equivalent:
\begin{enumerate}[\rm (1)]
	\item $X$ is locally compact.
	\item $\mk (X)$ is a continuous semilattice, and the upper Vietoris topology and the Scott topology on $\mk (X)$ agree.
	\item $\mk (X)$ is a continuous semilattice, and $\xi_X^\sigma : X \longrightarrow \Sigma~\!\! \mk (X)$, $x\mapsto\ua x$, is continuous.
    \item $\mk (X)$ is a continuous semilattice, and $X$ has property Q.
    \item $X$ is core compact.
\end{enumerate}
\end{theorem}
\begin{proof} (1) $\Rightarrow$ (2): By Theorem \ref{SoberLC=CoreC}, Lemma \ref{K(X) way below relation} and \cite[Lemma 7.26]{Schalk} (see Proposition \ref{LC sober domain V=S} below).

(2) $\Rightarrow$ (3):  By Remark \ref{xi embdding}.

(3) $\Rightarrow$ (1): For $x\in U\in \mathcal O(X)$, by Lemma \ref{WF V less S}, $\ua x\in \Box U\in \sigma (\mk (X))$, and hence by Lemma \ref{cont=CD}, there is $K\in \mk (X)$ with $\ua x\in \ii~\!_{\sigma (\mk (X))}\ua_{\mk (X)} K\subseteq \ua_{\mk (X)}K\subseteq \Box U$. Let $V=(\xi_{X}^{\sigma})^{-1}(\ii~\!_{\sigma (\mk (X))}\ua_{\mk (X)} K)$. Then by the continuity of $\xi_X^\sigma$, we have $V\in \mathcal O(X)$ and $x\in V\subseteq K\subseteq U$. Thus $X$ is locally compact.

(1) $\Rightarrow$ (4): By Lemma \ref{K(X) way below relation}.

(4) $\Rightarrow$ (1): Let $x\in U\in \mathcal O(X)$. Then by the continuity of $\mk (X)$, $\Downarrow_{\mk (X)}\ua x$ is directed (note that the order on $\mk (X)$ is the reverse inclusion order) and $\ua x=\bigvee_{\mk (X)} \Downarrow_{\mk (X)}\ua x$. It follows from Lemma \ref{sups in Smyth} that $\ua x=\bigvee_{\mk (X)} \Downarrow_{\mk (X)}\ua x=\bigcap \Downarrow_{\mk (X)}\ua x\subseteq U$, and hence by the well-filteredness of $X$, there is $K\in \Downarrow_{\mk (X)}\ua x$ such that $K\subseteq U$. Since $X$ has property {\rm Q}, we have $\ua x\subseteq \ii~\! K\subseteq K\subseteq U$. Therefore, $X$ is locally compact.

(1) $\Leftrightarrow$ (5): By Corollary \ref{WFcorcomp-sober} or Theorem \ref{SoberLC=CoreCNew}.

\end{proof}

Theorem \ref{main result 1} can be restated as the following one.

\begin{theorem}\label{main result 1+}  Let $X$ be a well-filtered space such that $\mk (X)$ is a continuous semilattice. Then the following conditions are equivalent:
\begin{enumerate}[\rm (1)]
	\item $X$ is locally compact.
	\item The upper Vietoris topology and the Scott topology on $\mk (X)$ agree.
	\item $\xi_X^\sigma : X \longrightarrow \Sigma~\!\! \mk (X)$, $x\mapsto\ua x$, is continuous.
    \item $X$ has property Q.
\end{enumerate}
\end{theorem}

\begin{corollary}\label{K(X) continuous implies LC redbook} \emph{(\cite[Proposition IV-2.19]{redbook})}  Let $X$ be a sober space having property Q. Then $\mk (X)$ is a continuous semilattice if{}f $X$ is locally compact.
\end{corollary}

The following example shows that for a well-filtered space $X$, when $X$ lacks property {\rm Q}, Theorem \ref{main result 1} may not hold. It also shows that the well-filteredness of $X$ and the continuity of $\mk (X)$ together do not imply the sobriety of $X$ in general.

\begin{example}\label{WF+K(X) cont not sober}
	Let $X$ be a uncountably infinite set and $X_{coc}$ the space equipped with \emph{the co-countable topology} (the empty set and the complements of countable subsets of $X$ are open). Then
\begin{enumerate}[\rm (a)]
    \item $\Gamma (X_{coc})=\{\emptyset, X\}\bigcup X^{(\leqslant\omega)}$ and $\ir (X_{coc})=\ir_c(X_{coc})=\{X\}\bigcup \{\{x\} : x\in X\}$.
    \item $\mk (X_{coc})=X^{(<\omega)}\setminus \{\emptyset\}$ and $\ii~\!K=\emptyset$ for all $K\in \mk (X_{coc})$.
    \item $\mk (X_{coc})$ is a dcpo and every element in $\mk (X_{coc})$ is compact. Hence $\mk (X_{coc})$ is an algebraic domain.
    \item $X_{coc}$ is a well-filtered $T_1$ space, but it not sober.
    \item The upper Vietoris topology and the Scott topology on $\mk ((X_{coc}))$ do not agree.
	\item $\xi_{X_{coc}}^\sigma : X_{coc} \longrightarrow \Sigma~\!\! \mk (X_{coc})$, $x\mapsto\ua x$, is not continuous.
    \item $X_{coc}$ does not have property {\rm Q}.
    \item $X_{coc}$ is not locally compact and not first countable.
  \end{enumerate}
\end{example}

The following example shows that even for a sober space $X$, when $X$ lacks property {\rm Q}, \cite[Proposition IV-2.19]{redbook}) (i.e., Corollary \ref{K(X) continuous implies LC redbook}) may not hold.

\begin{example}\label{example 1.25} (\cite[Example II-1.25]{redbook})  Let $p$ be a point in $\beta (\mathbb{N})\setminus \mathbb{N}$, where $\beta (\mathbb{N})$ is the Stone-C\v ech compactification of the discrete space of natural numbers, and consider on $X=\mathbb{N}\bigcup\{p\}$ the
induced topology. Then the space $X$ is a non-discrete Hausdorff space, and hence a sober space. Every compact subset of $X$ is finite. Therefore, $\mk (X)$ is an algebraic domain and $X$ is not locally compact. By Theorem \ref{main result 1}, $X$ does not have property {\rm Q} and is not core compact.  Furthermore, the upper Vietoris topology and the Scott topology on $\mk (X)$ do not agree, and the mapping $\xi_X^\sigma : X \longrightarrow \Sigma~\!\! \mk (X)$, $x\mapsto\ua x$, is not continuous.
\end{example}

By Example \ref{WF+K(X) cont not sober}, Theorem \ref{main result 1} (or Theorem \ref{main result 1+}) strengthens \cite[Proposition IV-2.19]{redbook}) and shows that the converse of \cite[Proposition IV-2.19]{redbook}) holds in the following sense: for a well-filtered space $X$ such that K(X) is a continuous semilattice, $X$ is locally compact iff $X$ has property Q.

\begin{lemma}\label{xi Scott cont} For a dcpo $P$, $\xi_{\Sigma~\!\! P}^\sigma : \Sigma~\!\! P \longrightarrow \Sigma~\!\! \mk (\Sigma~\!\! P)$ is continuous.
\end{lemma}
\begin{proof} For any $D\in \mathcal D(P)$, we have $\xi_{\Sigma~\!\! P}^\sigma (\bigvee D)=\ua \bigvee D=\bigcap_{d\in D}\ua d=\bigvee_{\mk (\Sigma~\!\! P)}\xi_{\Sigma~\!\! P}^\sigma(D)$ by Lemma \ref{sups in Smyth}. So $\xi_{\Sigma~\!\! P}^\sigma : \Sigma~\!\! P \longrightarrow \Sigma~\!\! \mk (\Sigma~\!\! P)$ is continuous.
\end{proof}

We get the following corollary from Theorem \ref{main result 1} and Lemma \ref{xi Scott cont}.

\begin{corollary}\label{K(X) continuous implies LC redbook+1} For a dcpo $P$ having the well-filtered Scott topology, the following conditions are equivalent:
\begin{enumerate}[\rm (1)]
	\item $\Sigma~\!\! P$ is locally compact.
	\item $\mk (\Sigma~\!\! P)$ is a continuous semilattice, and the upper Vietoris topology and the Scott topology on $\mk (\Sigma~\!\! P)$ agree.
    \item $\mk (\Sigma~\!\! P)$ is a continuous semilattice, and $\Sigma~\!\! P$ has property Q.
    \item $\mk (\Sigma~\!\! P)$ is a continuous semilattice.
    \item $\Sigma~\!\! P$ is core compact.
\end{enumerate}
\end{corollary}

\begin{definition}\label{semiclosed} Let $P$ be a poset equipped with a topology $\tau$.
\begin{enumerate}[\rm (1)]
\item $(P, \tau)$ is called \emph{upper semicompact}, if $\ua x$ is compact for any $x\in P$, or equivalently, if $\ua x\bigcap A$ is compact for any $x\in P$ and $A\in \Gamma ((P, \tau))$.
   \item $(P, \tau)$ is called \emph{weakly upper semicompact} if $\ua x\bigcap A$ is compact for any $x\in P$ and $A\in \ir_c((P, \tau))$.
    \end{enumerate}
\end{definition}

\begin{lemma}\label{Scott is WF1}\emph{(\cite{xuzhao})} For a dcpo $P$, if $(P, \lambda (P))$ is weakly upper semicompact \emph{(}especially, if $(P, \lambda (P))$ is upper semicompact or $P$ is bounded complete\emph{)}, then $(P, \sigma (P))$ is well-filtered.
\end{lemma}

By Corollary \ref{K(X) continuous implies LC redbook} and Lemma \ref {Scott is WF1}, we get the following corollary.

\begin{corollary}\label{K(X) continuous implies LC redbook+2}  For a dcpo $P$, if $(P, \lambda (P))$ is weakly upper semicompact \emph{(}especially, if $(P, \lambda (P))$ is upper semicompact or $P$ is bounded complete\emph{)}, then the following conditions are equivalent:
\begin{enumerate}[\rm (1)]
	\item $\Sigma~\!\! P$ is locally compact.
	\item $\mk (\Sigma~\!\! P)$ is a continuous semilattice, and the upper Vietoris topology and the Scott topology on $\mk (\Sigma~\!\! P)$ agree.
    \item $\mk (\Sigma~\!\! P)$ is a continuous semilattice, and $\Sigma~\!\! P$ has property Q.
    \item $\mk (\Sigma~\!\! P)$ is a continuous semilattice.
    \item $\Sigma~\!\! P$ is core compact.
\end{enumerate}
\end{corollary}

\section{First-countability of Smyth power spaces}

Now we consider the following question: for a first-countable (resp., second-countable) space $X$, does its Smyth power space $P_S(X)$ be first-countable (resp., second-countable)?

First, we have the following result, which was indicated in the proof of \cite[Proposition 6]{Brecht}.

\begin{theorem}\label{Smyth CII}  For a $T_0$ space, the following two conditions are equivalent:
\begin{enumerate}[\rm (1)]
\item $X$ is second-countable.
\item $P_S(X)$ is second-countable.
\end{enumerate}
\end{theorem}

\begin{proof}  (1) $\Rightarrow$ (2): Let $\mathcal{B}\subseteq O(X)$ be a countable base of $X$ and let $\mathcal{B}_S=\{\Box \bigcup\limits_{i=1}^n U_i: n\in \mathbb{N}\mbox {~and~} U_i\in \mathcal{B}\mbox{~for all~} 1\leq i\leq n\}$. Then $\mathcal{B}_S$ is countable. Now we show that $\mathcal{B}_S$ is a base of $P_S(X)$. Let $K\in \mk (X)$ and $U\in \mathcal O(X)$ with $K\in \Box U$. Then for each $k\in K$, there is $U_k\in \mathcal{B}$ with $k\in U_k\subseteq U$. By the compactness of $K$, there is a finite subset $\{k_1, k_2, ..., k_m\}\subseteq K$ such that $K\subseteq V=\bigcup\limits_{i=1}^m U_{k_i}\subseteq U$, and hence $\Box V\in \mathcal B_S$ and $K\in \Box V\subseteq \Box U$. Thus $\mathcal{B}_S$ is a base of $P_S(X)$, proving that $P_S(X)$ is second-countable.

(2) $\Rightarrow$ (1): As a subspace of $P_S(X)$, $P_S(\mathcal S^u(X))$ is second-countable, and hence $X$ is second-countable since $X$ is homeomorphic to $P_S(\mathcal S^u(X))$.

\end{proof}

 Next, we consider the first-countability. Since the first-countability is a hereditary property and any $T_0$ space $X$ is homeomorphic to $P_S(\mathcal S^u(X))$, a subspace of $P_S(\mathcal S(X))$, we have the following result.

\begin{proposition}\label{Smyth CI}  Let $X$ be a $T_0$ space. If $P_S(X)$ is first-countable, then $X$ is first-countable.
\end{proposition}

Consider in the plane $\mathbb{R}^2$ two concentric circles $C_i=\{(x,y)\in \mathbb{R}^2 :
x^2+y^2=i\}$, where $i=1, 2$, and their union $X=C_1\bigcup C_2$; the projection of $C_1$ onto $C_2$ from
the point $(0,0)$ is denoted by $p$. On the set $X$ we generate a topology by defining
a neighbourhood system $\{B(z): z\in X\}$ as follows: $B(z)=\{{z}\}$ for $z\in C_2$ and
$B(z)=\{U_j(z): j\in \mathbb{N}\}$ for $z\in C_1$, where $U_j=V_j\bigcup p(V_j\setminus \{z\})$ and $V_j$ is is the arc of $C_1$ with center at $z$ and of length $1/j$. The space $X$ is called the \emph{Alexandroff double circle} (see \cite[Example 3.1.26]{Engelking}).

\begin{proposition}\label{Alexandroff double circle property} \emph{(\cite{Engelking})} Let $X$ be the Alexandroff double circle. Then
\begin{enumerate}[\rm (1)]
\item  $X$ is Hausdorff and first-countable.
 \item  $X$ is not separable, and hence not second-countable.
 \item  $X$ is compact and locally compact.
 \item $C_1$ is a compact subspace of $X$.
 \item $C_2$ is a discrete subspace of $X$.
 \end{enumerate}
\end{proposition}

The following example shows that the converse of Proposition \ref{Smyth CI} fails in general.

\begin{example}\label{X CI but Smyth not} Let $X=C_1\bigcup C_2$ be the Alexandroff double circle. Then by Proposition \ref{Alexandroff double circle property}, $X$ is a compact Hausdorff first-countable space and $C_1\in \mk(X)$. Now we prove that $P_S(X)$ is not first-countable. First, for any open subset $U\in \mathcal O(X)$ with $C_1\subseteq U$, there is a family $\{U_j=V_{n(j)}\bigcup p(V_{n(j)}\setminus \{z_j\}) : j\in J\}$ of basic open sets such that $C_1\subseteq \bigcup_{j\in J} U_j\subseteq U$, where $V_{n(j)}$ is the arc of $C_1$ with center at $z_j$ and of length $1/n(j)$, and $p$ is the projection of $C_1$ onto $C_2$ from
the point $(0,0)$. By the compactness of $C_1$, there is a finite set $\{z_{j_1}, z_{j_2}, ..., z_{j_n}\}\subseteq C_1$ such that $C_1\subseteq\bigcup\limits_{i=1}^n U_{j_i}\subseteq U$, and hence $C_2\setminus U\subseteq \{p(z_{j_1}), p(z_{j_2}), ..., p(z_{j_n})\}$. Thus $C_2\setminus U$ is finite. Suppose that $\{W_n : n\in \mathbb{N}\}$ is a countable family of open sets containing $C_1$. Then $C_2\setminus \bigcap_{n\in \mathbb{N}} W_n=\bigcup_{n\in \mathbb{N}} (C_2\setminus W_n)$ is countable. Choose $x\in C_2\bigcap\bigcap_{n\in \mathbb{N}} W_n$ and let $V=X\setminus \{x\}$. Then $C_1\subseteq V\in O(X)$ but $W_n\nsubseteq V$ for all $n\in \mathbb{N}$. Thus there is no countable base at $C_1$ in $P_S(X)$, proving that $P_S(X)$ is not first-countable.
\end{example}

\begin{theorem}\label{min Compact countable is Smth CI} Let $X$ be a first-countable $T_0$ space. If $\mathrm{min}(K)$ is countable for any $K\in \mk (X)$, then
$P_S(X)$ is first-countable.
\end{theorem}
\begin{proof} For each $x\in X$, by the first-countability of $X$, there exists a countable base $\mathcal{B}_x$ at $x$. Let $K\in \mk (X)$. Then by  assumption $\mathrm{min} (K)$ is countable. Let $\mathcal{B}_K=\{\Box \bigcup_{c\in C}\varphi (c) : C\in \mathrm{min} (K)^{(<\omega)} \mbox{~and~}\varphi\in\prod\limits_{c\in C}\mathcal{B}_c\}$. Then $\mathcal{B}_K$ is countable. Now we show that $\mathcal{B}_K$ is a base at $K$.
Suppose that $U\in \mathcal O(X)$ and $K\in \Box U$. Then $\mathrm{min} (K)\subseteq K\subseteq U$. For each $k\in \mathrm{min}(K)$, there is a $\psi(k)\in \mathcal{B}_k$ with $k\in \psi(k)\subseteq U$.  By the compactness of $\mathrm{min}(K)$, there is a finite set $\{k_1, k_2, ..., k_m\}\subseteq \mathrm{min}(K)$ such that $\mathrm{min}(K)\subseteq \bigcup\limits_{i=1}^{m}\psi(k_i)\subseteq U$. Let $V=\bigcup\limits_{i=1}^{m}\psi(k_i)$. Then $K\subseteq V\subseteq U$. It follows that $\Box V\in \mathcal{B}_K$ and $K\in \Box V\subseteq \Box U$, proving that $\mathcal{B}_K$ is a base at $K$. Thus $P_S(X)$ is first-countable.
\end{proof}

\begin{corollary}\label{Compact countable is Smth CI} Let $X$ be a first-countable $T_0$ space. If all compact subsets of $X$ are countable, then
$P_S(X)$ is first-countable.
\end{corollary}

\begin{proposition}\label{metric space is Smth CI} For a metric space $(X, d)$, $P_S((X, d))$ is first-countable.
\end{proposition}
\begin{proof} For $K\in \mk ((X, d))$, let $\mathcal B_K=\{B(K, 1/n) : n\in \mathbb{N}\}$. Then by Proposition \ref{metric space compact sets}, $\mathcal B_K=\{B(K, 1/n) : n\in \mathbb{N}\}$ is a countable base at $K$ in $P_S((X, d))$. Thus $P_S((X, d))$ is first-countable.

\end{proof}

\section{Coincidence of the upper Vietoris topology and Scott topology}

For a well-filtered space $X$, from Theorem \ref{main result 1} we know that it is an important property that the upper Vietoris topology agrees with the Scott topology on $\mk (X)$. In this section we investigate the conditions under which the upper Vietoris topology coincides with the Scott topology on $\mk (X)$.

\begin{proposition}\label{Heckmann Scott=V} \emph{(\cite{Heckmann})} Let $P$ be a dcpo. If $\Sigma ~\!\! P$ is well-filtered and locally compact, then the upper Vietoris topology agrees the Scott topology on $\mk (\Sigma ~\!\! P)$.
\end{proposition}

From Lemma \ref{quasic=LHC} and Proposition \ref{Heckmann Scott=V} we get the following result.

\begin{proposition}\label{quasicont domain V=S} \emph{(\cite{Klause-Heckmann})} For a quasicontinuous domain $P$, the upper Vietoris topology agrees with the Scott topology on $\mk (\Sigma~\!\! P)$.
\end{proposition}

For a general $T_0$ space $X$, Schalk \cite{Schalk} proved the following result.

\begin{proposition}\label{LC sober domain V=S} \emph{(\cite{Schalk})} If $X$ is a locally compact sober space, then the upper Vietoris topology and the Scott topology on $\mk (X)$ coincide.
\end{proposition}

By Theorem \ref{SoberLC=CoreCNew} and Proposition \ref{LC sober domain V=S}, we have the following corollary.

\begin{corollary}\label{LC WF domain V=S}  If $X$ is a core compact well-filtered space, then the upper Vietoris topology and the Scott topology on $\mk (X)$ coincide.
\end{corollary}

\begin{theorem}\label{Smyth CII V=S} \emph{(\cite{Brecht})} If $X$ is a second-countable sober space, then the upper Vietoris topology and the Scott topology on $\mk (X)$ coincide.
\end{theorem}

From Theorem \ref{main result 1}, Theorem \ref{CI+WF is sober} and Theorem \ref{Smyth CII V=S}, we directly deduce the following two results.

\begin{corollary}\label{CII+WF V=S}  If $X$ is a second-countable well-filtered space, then the upper Vietoris topology and the Scott topology on $\mk (X)$ coincide.
\end{corollary}

\begin{corollary}\label{CII+WF LC} For a second-countable well-filtered space $X$, the following conditions are equivalent:
\begin{enumerate}[\rm (1)]
	\item $X$ is locally compact.
    \item $\mk (X)$ is a continuous semilattice, and $\xi_X^\sigma : X \longrightarrow \Sigma~\!\! \mk (X)$, $x\mapsto\ua x$, is continuous.
    \item $\mk (X)$ is a continuous semilattice, and $X$ has property Q.
    \item $\mk (X)$ is a continuous semilattice.
    \item $X$ is core compact.
\end{enumerate}
\end{corollary}

In \cite{Hofmann-Lawson} (see \cite[Exercise V-5.25]{redbook}), Hofmann and Lawson constructed a second-countable core compact $T_0$ space $X$ in which every compact subset has empty interior. So $X$ is not locally compact and does not have property Q. By Theorem \ref{Smyth CII}, $P_S(X)$ is second-countable; and by  Corollary \ref{WFcorcomp-sober} or Corollary \ref{CII+WF LC}, $X$ is not well-filtered.

Now we give another main result of this paper.

\begin{theorem}\label{Smyth CI S=V}  If $X$ is a well-filtered space and $P_S(X)$ is first-countable, then the upper Vietoris topology agrees with and the Scott topology on $\mk (X)$.
\end{theorem}

\begin{proof}  By Lemma \ref{WF V less S}, $\mathcal O(P_S(X))\subseteq \sigma(\mk (X))$. Now we show that $\mathcal O(P_S(X))\supseteq \sigma(\mk (X))$. Assume $K\in \mathcal U\in O(P_S(X))$. Since $P_S(X)$ is first-countable and $\{\Box V : V\in \mathcal O(X)\}$ is a base of $P_S(X)$, we have a countable family $\{U_n : n\in \mathbb{N}\}\subseteq \mathcal O(X)$ such that $\{\Box U_n : n\in \mathbb{N}\}$ is a base at $K$ in $P_S(X)$. We can assume that $U_1\supseteq U_2\supseteq ... \supseteq U_n \supseteq U_{n+1} \supseteq ...$ (otherwise, we replace $U_n$ with $\bigcap_{i=1}^n U_i$ for each $n\in \mathbb{N}$). We claim that $\Box U_n\subseteq \mathcal U$ for some $n\in \mathbb{N}$. Assume, on the contrary, that $\Box U_m\nsubseteq \mathcal U$ for all $m\in \mathbb{N}$. For each $m\in \mathbb{N}$, choose $K_m\in \Box U_m\setminus \mathcal U$, and let $G_m=K\bigcup \bigcup_{n\geq m}K_n$.

{Claim 1:} $G_m\in \mk (X)$ for each $m\in \mathbb{N}$.

Suppose that $\{W_j : j\in J\}$ is an open cover of $G_m$ and let $W_J=\bigcup_{j\in J}W_j$. Then $W_J\in \mathcal O(X)$ and $K\subseteq W_J$. By the compactness of $K$, there is $J_1\in J^{(<\omega)}$ such that $K\subseteq W_{J_1}=\bigcup_{j\in J_1}W_j$. Since $\{\Box U_n : n\in \mathbb{N}\}$ is a base at $K$ in $P_S(X)$, there is $n_o\in \mathbb{N}$ such that $K\in \Box U_{n_0}\subseteq \Box W_{J_1}$, that is, $K\subseteq U_{n_0}\subseteq W_{J_1}$. As $(U_n)_{n\in \mathbb{N}}$ is a decreasing sequence, we have that $K_l\subseteq U_l\subseteq U_{n_0}\subseteq W_{J_1}$ for all $l\geq n_0$. By the compactness of $\bigcup\limits_{i=m}^{n_0-1} K_i$, there is $J_2\in J^{(<\omega)}$ such that $\bigcup\limits_{i=m}^{n_0-1} K_i\subseteq W_{J_2}=\bigcup_{j\in J_2}W_j$. Therefore, $G_m=(K\bigcup \bigcup\limits_{n\geq n_0}K_n) \bigcup \bigcup\limits_{i=m}^{n_0-1} K_i\subseteq W_{J_1}\bigcup W_{J_2}$. Thus $G_m\in \mk (X)$.

{Claim 2:} $G_m \supseteq G_{m+1}$ for each $m\in \mathbb{N}$.

{Claim 3:}  $K=\bigcap_{m\in \mathbb{N}}G_m$.

Clearly, $K\subseteq \bigcap_{m\in \mathbb{N}}G_m$. Conversely, assume $x\not\in K$. Then $K\in \Box (X\setminus \da x)$, and whence there is $m_0\in \mathbb{N}$ such that $K\in \Box U_{m_0}\subseteq \Box (X\setminus \da x)$. It follows that $x\not\in G_m$ for all $m\geq m_0$. Hence $x\not\in \bigcap_{m\in \mathbb{N}}G_m$. Therefore, $K=\bigcap_{m\in \mathbb{N}}G_m$.

By the above three claims and Lemma \ref{sups in Smyth}, $K=\bigvee_{\mk (X)}\{G_m : m\in \mathbb{N}\}\in \mathcal U\in \sigma(\mk (X))$, and hence $G_q\in \mathcal U$ for some $q\in \mathbb{N}$. But then $K_n\in \mathcal U$ for all $n\geq q$, a contradiction.

Therefore, $K\in \Box U_n\subseteq \mathcal U$ for some $n\in \mathbb{N}$, and consequently, $\mathcal U\in \mathcal O(P_S(X))$. It is thus proved that the upper Vietoris topology and the Scott topology
on $\mk (X)$ coincide.
\end{proof}

By Theorem \ref{min Compact countable is Smth CI} and Theorem \ref{Smyth CI S=V}, we get the following result.

\begin{corollary}\label{min compact countable CI S=V}  Let $X$ be a first-countable well-filtered space.
If $\mathrm{min}(K)$ is countable for any $K\in \mk (X)$, then the upper Vietoris topology and the Scott topology on $\mk (X)$ coincide.
\end{corollary}

\begin{corollary}\label{countable countable CI S=V}  For a first-countable well-filtered space $X$ in which
all compact subsets are countable, the upper Vietoris topology and the Scott topology on $\mk (X)$ agree.
\end{corollary}

By Theorem \ref{main result 1}, Theorem \ref{Smyth CI S=V} and Corollary \ref{min compact countable CI S=V}, we have the following three corollaries.

\begin{corollary}\label{Smyth CI LC} Let $X$ be a well-filtered space for which $P_S(X)$ is first-countable. Then the following conditions are equivalent:
\begin{enumerate}[\rm (1)]
	\item $X$ is locally compact.
    \item $\mk (X)$ is a continuous semilattice, and $\xi_X^\sigma : X \longrightarrow \Sigma~\!\! \mk (X)$, $x\mapsto\ua x$, is continuous.
    \item $\mk (X)$ is a continuous semilattice, and $X$ has property Q.
    \item $\mk (X)$ is a continuous semilattice.
    \item $X$ is core compact.
\end{enumerate}
\end{corollary}

\begin{corollary}\label{min compact countable CI LC} Let $X$ be a first-countable well-filtered space for which $\mathrm{min}(K)$ is countable for any $K\in \mk (X)$. Then the following conditions are equivalent:
\begin{enumerate}[\rm (1)]
	\item $X$ is locally compact.
    \item $\mk (X)$ is a continuous semilattice, and $\xi_X^\sigma : X \longrightarrow \Sigma~\!\! \mk (X)$, $x\mapsto\ua x$, is continuous.
    \item $\mk (X)$ is a continuous semilattice, and $X$ has property Q.
    \item $\mk (X)$ is a continuous semilattice.
    \item $X$ is core compact.
\end{enumerate}
\end{corollary}

\begin{corollary}\label{compact countable CI LC} Let $X$ be a first-countable well-filtered space for which all compact subsets of $X$ are countable. Then the following conditions are equivalent:
\begin{enumerate}[\rm (1)]
	\item $X$ is locally compact.
    \item $\mk (X)$ is a continuous semilattice, and $\xi_X^\sigma : X \longrightarrow \Sigma~\!\! \mk (X)$, $x\mapsto\ua x$, is continuous.
    \item $\mk (X)$ is a continuous semilattice, and $X$ has property Q.
    \item $\mk (X)$ is a continuous semilattice.
    \item $X$ is core compact.
\end{enumerate}
\end{corollary}

By Proposition \ref{metric space property}, Lemma \ref{metric space is Smth CI} and Theorem \ref{Smyth CI S=V}, we get the following two results.

\begin{corollary}\label{metric space S=V}  For a metric space $(X, d)$, the upper Vietoris topology coincides with the Scott topology on $\mk ((X, d))$.
\end{corollary}

\begin{corollary}\label{metric space LC} For a metric space $(X, d)$, the following conditions are equivalent:
\begin{enumerate}[\rm (1)]
	\item $(X, d)$ is locally compact.
    \item $\mk (X)$ is a continuous semilattice, and $\xi_{(X, d)}^\sigma : (X, d) \longrightarrow \Sigma~\!\! \mk ((X, d))$, $x\mapsto\{x\}$, is continuous.
    \item $\mk ((X, d))$ is a continuous semilattice, and $(X, d)$ has property Q.
    \item $\mk ((X, d))$ is a continuous semilattice.
    \item $(X, d)$ is core compact.
\end{enumerate}
\end{corollary}

Let $\mathbb{R}$ be the set of all real numbers. $\mathbb{R}$ endowed with the topology taking the family $\{[x, y) : x<y\}$ as a base is called the \emph{Sorgenfrey line} and denoted by $\mathbb{R}_l$. Dually, we endow $\mathbb{R}$ with the topology generated by $\{(x, y] : x<y\}$ as a base, and denote the resulting space by $\mathbb{R}_r$. A subset $A\subseteq \mathbb{R}_l$ is called \emph{bounded} if $A\subseteq [-n, n]$ for some $n\in \mathbb{N}$. As one of the "universal counterexamples" in
general topology, $\mathbb{R}_l$ poses many important topological properties (cf. \cite{Engelking, Jean-2013}). In particular, the Sorgenfrey line has the following properties (cf. \cite{Engelking}).

\begin{proposition}\label{Sorgenfrey line property}
\begin{enumerate}[\rm (1)]
\item  $\mathbb{R}_l$ is perfectly normal, first-countable and separable.
 \item  $\mathbb{R}_l$ is not second-countable.
 \item  $\mathbb{R}_l$ is neither compact nor locally compact.
 \item Every compact subset of $\mathbb{R}_l$ is countable.
 \end{enumerate}
\end{proposition}

\begin{lemma}\label{compact set in Sorgenfrey line} \emph{(\cite{xu90})} For a subset $A$ of $\mathbb{R}_l$, the following two conditions are equivalent:
\begin{enumerate}[\rm (1)]
\item $A$ is compact in $\mathbb{R}_l$.
\item $A$ is a bounded closed subset of $\mathbb{R}_l$, and $A$ has no accumulation point in $\mathbb{R}_r$ \emph{(}that is, there is no point $x\in \mathbb{R}$ such that $x\in \cl_{\mathbb{R}_r} (A\setminus \{x\})$\emph{)}.
\end{enumerate}
\end{lemma}

\begin{example}\label{Sorgenfrey line S=V} Consider the Sorgenfrey line $\mathbb{R}_l$. Then by Theorem \ref{main result 1}, Corollary \ref{Compact countable is Smth CI}, Theorem \ref{Smyth CI S=V}, Proposition \ref{Sorgenfrey line property} and Lemma \ref{compact set in Sorgenfrey line}, we have
\begin{enumerate}[\rm (1)]
\item $\mathbb{R}_l$ is first-countable and Hausdorff, and hence sober.
\item $P_S(\mathbb{R}_l)$ is first-countable.
\item $\ii K=\emptyset $ for any $K\in \mk (\mathbb{R}_l)$, and whence $P_S(\mathbb{R}_l)$ is not locally compact.
\item the upper Vietoris topology and the Scott topology on $\mk (\mathbb{R}_l)$ agree.
\item $K_1 \not\ll K_2$ for any $K_1, K_2\in \mk (\mathbb{R}_l)$, so $\mk (\mathbb{R}_l)$ is not continuous and $\mathbb{R}_l$ has property {\rm Q}.
\item $\xi_{\mathbb{R}_l}^\sigma : \mathbb{R}_l \longrightarrow \Sigma~\!\! \mk (\mathbb{R}_l)$, $x\mapsto \{x\}$, is continuous.
\item $\mathbb{R}_l$ is not core compact.
\end{enumerate}
\end{example}

\begin{example}\label{example 1.25 V not=S} Let $X$ be the space in Example \ref{example 1.25}. Then we have (a) $|X|=\omega$ and $X$ is sober; (b) $\mk (X)=X^{(<\omega)}\setminus \{\emptyset\}$; (c) $\mk (X)$ is an algebraic semilattice; and (d) the upper Vietoris topology and the Scott topology on $\mk (X)$ does not coincide or, equivalently, $\sigma (\mk (X))\nsubseteq \mathcal O(P_S(X))$. By Proposition \ref{Smyth CI} (or Theorem \ref{Smyth CI S=V}) and Corollary \ref{countable countable CI S=V}, Neither $P_S(X)$ nor $X$ is first-countable (cf. \cite[Corollary 3.6.17]{Engelking}).
\end{example}

Finally, we pose the following question.

\begin{question}\label{CI WF S=V?}  For a first-countable well-filtered space $X$, does the upper Vietoris topology and the Scott topology on $\mk (X)$ coincide?
\end{question}

\end{document}